\documentclass[12pt]{article}
 \usepackage{a4, latexsym, amsfonts}
 \newtheorem{theorem}{Theorem}[section]
\newtheorem{lemma}{Lemma}[section]

 \newtheorem{proposition}{Proposition}[section]

\usepackage{amsfonts}
\usepackage{amsmath}
\usepackage{amssymb}

 \newenvironment{proof}{\trivlist
      \item[\hskip\labelsep
      {\itshape Proof.}]\normalfont}
      {\hspace*{\fill}$\Box$\endtrivlist}
\begin{document}

%-------------------------------------------------------------------------
\vspace*{.2in}
\begin{center}{\large\bf On the transcendence of some infinite sums}\\\vspace*{.2in}
{\normalsize  Pingzhi Yuan$^*$\ Juan Li}

\end{center}

 \footnotetext { $*$ This author is responsible for  communications, and  supported by
  the Guangdong Provincial Natural Science Foundation (No. 8151027501000114) and NSF of China (No. 10571180).}

\vspace*{.1in}

\begin{abstract} In this paper we investigate the infinite convergent
sum $T=\sum_{n=0}^\infty\frac{P(n)}{Q(n)}$, where
  $P(x)\in\overline{\mathbb{Q}}[x]$, $Q(x)\in\mathbb{Q}[x]$ and $Q(x)$
  has only simple rational zeros. N.  Saradha and R.  Tijdeman have obtained sufficient and
  necessary conditions for the transcendence of $T$ if the degree
  of $Q(x)$ is 3. In this paper we give sufficient and
  necessary conditions for the transcendence of $T$ if the degree
  of $Q(x)$ is 4 and $Q(x)$ is reduced.

\vspace*{5pt}

 \vspace*{5pt}
 {\normalsize\bf MCS: } primary 11J81; secondary 11J86,11J91
 \end{abstract}

\section{\large Introduction }\label{sec_1}

\noindent In this paper we will investigate the transcendence of the
infinite convergent sum
$$T=\sum_{n=0}^\infty\frac{P(n)}{Q(n)},$$
where $P(x)\in\overline{\mathbb{Q}}[x]$, $Q(x)\in\mathbb{Q}[x]$ and
$Q(x)$
  has only simple rational zeros.   Owing to the reduction procedure
described in Tijdeman \cite{ti00,ti03}, we have $$T=A+S, \quad
S=\sum_{n=1}^\infty\frac{f(n)}{n},$$  where
$A\in\overline{\mathbb{Q}}$,  we take $q>1$ to be a positive integer
and $f(x)$ is  a number theoretic function which is periodic mod $q$
with $\sum_{i=1}^{q}f(i)=0$, which we will assume throughout the
paper.

About forty years ago,  Chowla \cite{ch52} and Erd\H os (see
\cite{li65}) formulated some  conjectures related to whether there
exists a rational-valued function $f(n)$ periodic with prime period
$p$ such that $\sum_{n=1}^\infty\frac{f(n)}{n}=0.$ One of the
conjectures was proved by Baker, Birch and Wirsing \cite{bbw73} in
1973. They used Baker's theory on linear forms in logarithms to
establish that $S\neq0$ if $f(n)$ is a non-vanishing function
defined on the integers with rational values and period $q$ such
that

i) $f(r)=0,\  \mathrm{if}\ 1<\mathrm{gcd}(r,q)<q$,

ii) the cyclotomic polynomial $\Phi_q$ is irreducible over $
\mathbb{Q}(f(1),\cdots,f(q))$.

\noindent They further showed that their result would be false if i)
or ii)
  is omitted (see \cite{bbw73}).

  In 1982, T. Okada \cite{ok82} established a result which provides a description of all functions for which ii) holds and $S = 0$.
  Okada's proof depends on the basic result on the linear independence
  of the logarithms of algebraic numbers and on the non-vanishing of
  $L(1,\chi)=\sum_{n=1}^{\infty}\frac{\chi(n)}{n}$ if $\chi$ is a non-principal Dirichlet character. The
  precise result is stated in Section 2.

    In 2001, S.D. Adhikari, N. Saradha, T.N. Shorey and
R. Tijdeman \cite{asst01} proved that if $S\neq0$, then $S$ is
transcendental. They used this result to prove that if $P(x)\in
\overline{\mathbb{Q}}[x]$ and $Q(x)\in\mathbb{Q}[x]$, where $Q(x)$
is a polynomial with simple rational roots which are all in the
interval $[-1, 0)$, then the infinite convergent sum
$T=\sum_{n=0}^\infty\frac{P(n)}{Q(n)}$ is $0$ or transcendental.
Further, if $Q(x)$ is a polynomial with simple rational roots, then
$T$ is a computable rational number or a transcendental number. For
more information on the developments sketched above we refer to
\cite{ad00} and \cite{ti00, ti03}. In particular, if the degree of
$Q(x)$ is 2, then
$$T=\sum_{n=0}^\infty\frac{\alpha }{(qn+s_1)(qn+s_2)}$$with
 $q, \ s_1,\ s_2$ integers, $\alpha\in\overline{\mathbb{Q}}$
 nonzero,
is transcendental if and only if $s_1\not\equiv s_2\ (\mathrm{mod}\
q)$. On the other hand, by above results, it is easy to see that
$$\sum_{n=0}^\infty\frac{1}{(3n+1)(3n+2)(3n+3)}> 0$$ and
$$\sum_{n=0}^\infty\frac{1}{(n+1)(2n+1)(4n+1)}=\frac{\pi}{3}$$  are
transcendental. The second equality was also proven by Lehmer
\cite{lm01} in 1975.

     In 2003, N. Saradha and R. Tijdeman \cite{st03}
rephrased Okada's theorem so that it becomes a decomposition lemma
and gave sufficient and necessary conditions for the transcendence
of $T=\sum_{n=0}^\infty\frac{P(n)}{Q(n)}$ if the degree of $Q(x)$ is
3. They proved that
$$T=\sum_{n=0}^\infty\frac{\alpha n+\beta
}{(qn+s_1)(qn+s_2)(qn+s_3)}$$is transcendental if $s_1, s_2, s_3$
are not in the same residue class mod $q$.
 However, when the degree of $Q(x)$ is 4, the example
$$T=\sum_{n=0}^\infty
  \frac{16n^2+12n-1}{(4n+1)(4n+2)(4n+3)(4n+4)}=0$$ shows
  that the corresponding result is not valid.

The main purpose of the present paper is to give sufficient and
necessary conditions for the transcendence of $T$ if the degree of
$Q(x)$ is 4, that is
\begin{equation}
T=\sum_{n=0}^\infty
  \frac{\alpha n^2+\beta
  n+\gamma}{(qn+s_1)(qn+s_2)(qn+s_3)(qn+s_4)}
\end{equation}
where $\alpha, \beta, \gamma\in \overline{\mathbb{Q}}$,  $s_1, s_2,
s_3, s_4 $ are distinct integers. By the reduction procedure
described in Tijdeman \cite{ti00,ti03}, without loss of generality,
we may assume that $0<s_1, s_2, s_3, s_4\leq q$,
$\mathrm{gcd}(\alpha x^2+\beta
  x+\gamma, (qx+s_1)(qx+s_2)(qx+s_3)(qx+s_4))=1$ and $\mathrm{gcd}(s_{1},s_2,s_3, s_{4},
q)=1$ throughout the paper. The following  simple example shows how
the reduction procedure works,
$$
\sum_{n=0}^{\infty}\frac{1}{(2n+1)(2n+2)(2n+3)}=-\frac{1}{2}+\sum_{n=0}^{\infty}\{\frac{1}{2n+1}-\frac{1}{2n+2}\}
=-\frac{1}{2}+\sum_{n=0}^\infty\frac{1}{(2n+1)(2n+2)}.$$

In Section 2 we shall give some preliminaries  that will be useful
for our further discussions.   In Section 3 we prove the following
Theorem.

\begin{theorem} Let

$$T=\sum_{n=0}^\infty
  \frac{\alpha n^2+\beta
  n+\gamma}{(qn+s_1)(qn+s_2)(qn+s_3)(qn+s_4)}$$
where  $s_1, s_2, s_3, s_4 $ are distinct positive integers $ \leq q
$ and $\alpha, \beta, \gamma \in \overline{\mathbb{Q}}$. Suppose
 $\mathrm{gcd}(\alpha x^2+\beta
  x+\gamma, (qx+s_1)(qx+s_2)(qx+s_3)(qx+s_4))=1$ , $\mathrm{gcd}(s_{1},s_2,s_3, s_{4},
q)=1$ and $\Phi_q$ is irreducible over
  $\mathbb{Q}(\alpha,\beta, \gamma)$. Then $T$ is transcendental
  except when
\begin{equation}T=\sum_{n=0}^\infty
  \frac{16n^2+12n-1}{(4n+1)(4n+2)(4n+3)(4n+4)}=0\end{equation}
  or \begin{equation}T=
\sum_{n=0}^{\infty}\frac{36n^{2}+36n-1}{(6n+1)(6n+2)(6n+4)(6n+5)}=0.\end{equation}
  \end{theorem}

\section{\large Preliminaries}\label{sec_2}
\noindent In this section we shall introduce some notations and
state the related results that will be needed in the sequel. We
denote by $\varphi(n)$ the Euler function and $P$ the set of  all
primes dividing $q$. We call the polynomial $Q(x)$ reduced if
$Q(x)\in\mathbb{Q}[x]$ and it has only simple rational zeros which
are all in the interval $[-1, 0)$. We denote by $v_{p}(n)$ the
exponent to which $p|n$ for  any prime $p$ and $n\in \mathbb{Z}$. We
write
$$J=\{a\in \mathbb{Z}\ |\ 1\leq a\leq q,\ \mathrm{gcd}(a,q)=1\},$$
$$L=\{r\in \mathbb{Z}\ |\ 1\leq r\leq q,\ 1<\mathrm{gcd}(r,q)<q\},$$
and $$L^{'}=L\cup\{q\}.$$
 For $p\in P$ and $r\in L^{'}$, we
define $$P(r)=\{p\in P\ |\ v_{p}(r)\geq v_{p}(q)\}$$ and

\begin{eqnarray*}\label{eqn-ghf2}
 \varepsilon(r,p)= \left\{ \begin{array}{ll}
         v_{p}(q)+\frac{1}{p-1},  & p\in P(r),  \\
         v_{p}(r),  & \mbox{ otherwise } .
        \end{array}
\right.
\end{eqnarray*}
For  $r\in L^{'}$ and $a\in J$, we define
$$A(r,a)=\frac{1}{\mathrm{gcd}(r,q)}\prod_{p\in P(r)}(1-\frac{1}{p^{\varphi(q)}})^{-1}\sum_{n\in S(r)}\frac{\sigma(r,a,n)}{n},$$
where $$S(r)=\{\prod_{p\in P(r)}p^{\alpha(p)}\ |\
0\leq\alpha(p)<\varphi(q)\}$$ and

\begin{eqnarray*}\label{eqn-ghf2}
 \sigma(r,a,n)= \left\{ \begin{array}{ll}
         1,  & \mbox { if }\quad  r\equiv an\gcd(r,q)\pmod{q},  \\
         0,  & \mbox{ otherwise } .
        \end{array}
\right.
\end{eqnarray*}

\noindent{\bf Theorem A.} (Okada \cite{ok82}). {\it  If} $\Phi_q$
{\it is irreducible over}
  $\mathbb{Q}(f(1),\cdots, f(q))$, {\it then} $S=\sum_{n=1}^\infty\frac{f(n)}{n}=0$  {\it if and only
  if}
\begin{equation} f(a)+\sum_{r\in
L}f(r)A(r,a)+\frac{f(q)}{\varphi(q)}=0, \qquad \ for\ all\ a\in
J,\end{equation} {\it and}
\begin{equation} \sum_{r\in
L^{'}}f(r)\varepsilon(r,p)=0, \qquad \ for\ all\ p\in P.
\end{equation}

   N. Saradha and R. Tijdeman \cite{st03} estabished an
equivalent version of Theorem A.

 \begin{lemma} {\rm (Decomposition
Lemma \cite{st03} )}.  Let $\Phi_q$  be irreducible over
  $\mathbb{Q}(f(1),\cdots, f(q))$.   Let M be the set of positive
integers which are composed of prime factors of $q$.   Then
$S=\sum_{n=1}^\infty\frac{f(n)}{n}=0$   if and only if
\begin{equation}\sum_{m\in M}\frac{f(am)}{m}=0, \qquad
for\ all\ a\in J,
\end{equation}
   and
$$\sum_{r\in L^{'}}f(r)\varepsilon(r,p)=0, \qquad \ for\ all\ p\in
P.$$
\end{lemma}

As a consequence of Lemma 2.1, they derived the following result.

\begin{lemma} {\rm (\cite{st03})}   Let $\Phi_q$  be irreducible
over
  $\mathbb{Q}(f(1),\cdots, f(q))$.  Suppose $S=\sum_{n=1}^\infty\frac{f(n)}{n}=0$.
  Then
$$\sum_{n=1}^{\infty}\frac{f(kn)}{n}=0,\ for\ every \ k \ with  \ \mathrm{gcd}(k,q)=1.$$
\end{lemma}
 The following
result given by S.D. Adhikari, N. Saradha, T.N. Shorey and R.
Tijdeman \cite{asst01}   is essential for the transcendence of
$\sum_{n=0}^\infty\frac{P(n)}{Q(n)}$.

\noindent {\bf Theorem B.} (\cite{asst01}) {\it Let}
$P(x)\in\overline{\mathbb{Q}}[x]$,  {\it and let }
$Q(x)\in\mathbb{Q}[x]$ {\it  be reduced. If}
$$T=\sum_{n=0}^\infty\frac{P(n)}{Q(n)}$$
{\it converges, then } $T$ {\it is} $0$ {\it or transcendental.}
%\end{Theorem}

   When the degree of $Q(x)$ is
3, N. Saradha and R. Tijdeman \cite{st03} obtained necessary and
sufficient conditions for the transcendence of
$T=\sum_{n=0}^\infty\frac{P(n)}{Q(n)}$.

\noindent{\bf Theorem C.} (\cite{st03}) {\it Let }
$T=\sum_{n=0}^\infty\frac{\alpha n+\beta
}{(qn+s_1)(qn+s_2)(qn+s_3)}$, {\it where} $\alpha,\ \beta\in
\overline{\mathbb{Q}}$, {\it  and}
  $|\alpha|+|\beta|>0$. {\it  Let} $\Phi_q$ {\it be irreducible over}
  $\mathbb{Q}(\alpha,\beta)$ {\it  and } $s_{1}, s_{2}, s_{3}$ {\it be distinct integers such that }
  $qn+s_1,\ qn+s_2,\ qn+s_3$ {\it do not vanish for }$n\geq
  0$.
  {\it Assume that} $s_{1}, s_{2}, s_{3}$ {\it  are not in the same residue
  class} $\mathrm{mod}\ q$. {\it Further let} $s_{1}\not\equiv s_{2}\ (\mathrm{mod}\ q)$ {\it if }
  $\alpha s_{3}=\beta q$; $s_{1}\not\equiv s_{3}\ (\mathrm{mod}\ q)$ {\it  if} $\alpha s_{2}=\beta q$;
  $s_{2}\not\equiv s_{3}\ (\mathrm{mod}\ q)$ {\it if } $\alpha s_{1}=\beta q$.
 {\it  Then $T$ is transcendental.}

The following result in \cite{ks01} will be useful in Section 3. For
the convenience of the reader, we provide the sketch of a proof
suggested  by Frazer Jarvis.

\begin{lemma}   Let  $n, d$,  and $r$ be
integers such that  $n>1$, $d>0$, $d|n$,  and $\gcd(r,d)=1$,
 then there are precisely $\varphi(n)/\varphi(d)\geq \varphi(n/d)$
 numbers which are coprime to $n$
  in the  set $ S=\{r+td,  t=1,2,\cdots,\frac{n}{d} \}$.\end{lemma}

\begin{proof}For primes $p|d$ there is no condition, but for primes
$p|n$ but $p\not|d$, the congruence classes for $r+td$ are equally
distributed mod $p$, so that $\frac{p-1}{p}$ of the possible numbers
are prime to $p$. The Chinese Remainder Theorem gives an
independence result. Since there are $\frac{n}{d}$ numbers
considered, the number we seek is
$$\frac{n}{d}\cdot\prod_{p|n, p\not|d}(1-\frac{1}{p}),$$
and the result easily follows.
\end{proof}

\section{\large   Proof of Theorem 1.1}\label{sec_3}
\noindent

Let

$$T=\sum_{n=0}^{\infty}\frac{\alpha_{k}n^{k}+\alpha_{k-1}n^{k-1}+\cdots+\alpha_{0}}{(qn+r_{1})\cdots
(qn+r_{m})},$$ where $\alpha_{0},\alpha_{1}, \cdots,\alpha_{k}\in
\overline{\mathbb{Q}} $, $r_1, \cdots, r_{m}$ are distinct positive
integers  and  $k\leq m-2$. Our main purpose  is to consider the
transcendence of $T$. By the reduction procedure given in Tijdeman
\cite{ti00,ti03}, we may restrict ourselves to the case that

i) $r_1, \cdots, r_{m}\ \mathrm{are}\ \mathrm{distinct}\
\mathrm{positive}\ \mathrm{integers}\ \leq q, \ \mathrm{gcd}(r_1,
\cdots, r_{m}, q)=1,$

ii) $\
\mathrm{gcd}(\alpha_{k}x^{k}+\alpha_{k-1}x^{k-1}+\cdots+\alpha_{0},(qx+r_{1})\cdots
(qx+r_{m}))=1.$

\noindent Therefore we need only consider the case $T=0$ by Theorem
B, which we shall assume from now on. By partial fractions, we get
$$
T=\sum_{n=0}^{\infty}\{\frac{A_{1}}{qn+r_{1}}+\frac{A_{2}}{qn+r_{2}}+\cdots+\frac{A_{m}}{qn+r_{m}}\},$$
where $A_{1}, \cdots, A_{m}\in \mathbb{Q}(\alpha_{0},\alpha_{1},
\cdots,\alpha_{k})$ are all nonzero numbers with

$$A_1+A_2+\cdots+A_m=0.$$
We define $f(n)$ for $n\geq0$ as follows:

\begin{eqnarray*}\label{eqn-ghf2}
f(n) = \left\{ \begin{array}{ll}
         A_1,  & n\equiv r_1\pmod{q}, \\
         \cdots & \cdots \\
         A_m,  & n\equiv r_m\pmod{q}, \\
         0,  & \mbox{ otherwise }.
        \end{array}
\right.
\end{eqnarray*}
Then $f(n)$ is a periodic function with period $q$ taking only $m$
non-zero values $f(r_1),f(r_2),\cdots,f(r_m)$ with
$$f(r_{1})+f(r_{2})+\cdots+f(r_{m})=0$$ and
$$T=\sum_{n=1}^\infty\frac{f(n)}{n}=0.$$
It is easy to see that $\mathbb{Q}(\alpha_{0},\alpha_{1},
\cdots,\alpha_{k})=\mathbb{Q}(A_{1}, A_{2},\cdots, A_{m})$. If
$\Phi_q$ is irreducible over $\mathbb{Q}(\alpha_{0},\alpha_{1},
\cdots,\alpha_{k})$, then $\Phi_q$ is irreducible over
$\mathbb{Q}(f(1),\cdots,f(q))$,  so (4), (5) and (6) are valid by
Theorem A and Lemma 2.1. We have
\begin{proposition} Suppose

$$T=\sum_{n=0}^{\infty}\frac{\alpha_{k}n^{k}+\alpha_{k-1}n^{k-1}+\cdots+\alpha_{0}}{(qn+r_{1})\cdots
(qn+r_{m})}=0,$$  where $r_1, \cdots, r_{m}$ are distinct positive
integers $\leq q$, $k\leq m-2$, and $\alpha_{0},\alpha_{1},
\cdots,\alpha_{k}\in \overline{\mathbb{Q}} $. Suppose $
\mathrm{gcd}(r_1, \cdots, r_{m}, q)=1,$  and $\
\mathrm{gcd}(\alpha_{k}x^{k}+\alpha_{k-1}x^{k-1}+\cdots+\alpha_{0},(qx+r_{1})\cdots
(qx+r_{m}))=1$ and $\Phi_q$ is irreducible over
  $\mathbb{Q}(\alpha_{0}, \cdots,\alpha_{k})$. Then there exists an $r_i$ with $1\leq i\leq m $ such that $\gcd(r_i, q)>1$.
\end{proposition}

\begin{proof} By the above arguments, if all of $\{r_{1}, \cdots,
  r_{m}\}$ are coprime to $q$, then $f(r)=0$ for all $r\in L^{'}$.
  Applying (4) with $a\in J$ we have
$f(a)=0$ for all $a\in J,$ a contradiction.  This completes the
proof.\end{proof}

The main purpose of the present paper is to investigate the
transcendence of $T$ in the case that $m=4$, that is
$$T=\sum_{n=0}^\infty
  \frac{\alpha n^2+\beta
  n+\gamma}{(qn+r_1)(qn+r_2)(qn+r_3)(qn+r_4)}=
\sum_{n=0}^{\infty}\{\frac{A_1}{qn+r_{1}}+\frac{A_2}{qn+r_{2}}+\frac{A_3}{qn+r_{3}}+\frac{A_4}{qn+r_{4}}\},$$
and $f(n)$ is a periodic function with period $q$ taking only four
non-zero values $f(r_1)=A_1,f(r_2)=A_2,f(r_3)=A_3,f(r_4)=A_4$
satisfying
$$f(r_{1})+f(r_{2})+f(r_{3})+f(r_{4})=0\ \mathrm{and}
\ T=\sum_{n=1}^\infty\frac{f(n)}{n}=0.$$

 We divide the proof of  Theorem 1.1 into four
cases depending on the number $\rho$ of elements of  $\{r_1, r_2,
r_3, r_4 \}$ which are coprime to $q$. By Proposition 3.1, we have
$\rho\leq 3$. First suppose that $\rho=3$, then without loss of
generality we may assume that $\mathrm{gcd}(r_{1}, q)
>1$ and $\mathrm{gcd}(r_{2}r_{3}r_{4},q)=1$. If
$p|\mathrm{gcd}(r_{1}, q)$ and $p\nmid r_{i}$, $i=2,3,4$, then by
(5) we get $f(r_{1})\varepsilon(r_{1},p)=0$, and so $f(r_{1})=0$
since $\varepsilon(r_{1},p)\neq 0$, a contradiction. Consequently if
$T=\sum_{n=0}^\infty
  \frac{\alpha n^2+\beta
  n+\gamma}{(qn+r_1)(qn+r_2)(qn+r_3)(qn+r_4)}=0$ and
there exists an integer $r\in \{r_1, r_2, r_3, r_4 \}$ with
$\mathrm{gcd}(r,  q) >1$, then there exists at least another integer
$s \in \{r_1, r_2, r_3, r_4 \}\backslash\{r\}$ with $\mathrm{gcd}(r,
s, q)
>1$.

Now suppose $\rho=2$.  We have

\begin{proposition}
Suppose

$$T=\sum_{n=0}^\infty
  \frac{\alpha n^2+\beta
  n+\gamma}{(qn+r_1)(qn+r_2)(qn+r_3)(qn+r_4)}=0,$$
where  $r_1, r_2, r_3, r_4 $ are distinct positive integers $ \leq q
$ and $\alpha, \beta, \gamma \in \overline{\mathbb{Q}}$. Suppose
$\gcd(\alpha n^2+\beta
  n+\gamma, (qn+r_1)(qn+r_2)(qn+r_3)(qn+r_4))=1$,
$\mathrm{gcd}(r_{1}, r_{2},r_{3}, r_{4}, q)=1$ and $\Phi_q$ is
irreducible over
  $\mathbb{Q}(\alpha,\beta, \gamma)$. Suppose $\rho=2$.
  Then $T$ is transcendental except when
$$T=\sum_{n=0}^\infty
  \frac{16n^2+12n-1}{(4n+1)(4n+2)(4n+3)(4n+4)}$$
or $$T=
\sum_{n=0}^{\infty}\frac{36n^{2}+36n-1}{(6n+1)(6n+2)(6n+4)(6n+5)}.$$
\end{proposition}

\begin{proof}  Suppose $\rho=2$. Without loss of
generality we may assume that $\mathrm{gcd}(r_{1}, q)>1$,
$\mathrm{gcd}(r_{2}, q)>1$ and $\mathrm{gcd}(r_{3}r_{4}, q)=1$. By
the above arguments, we have $\mathrm{gcd}(r_{1}, r_{2}, q)=d >1$.

If $\varphi(d)>2$, we let $$a_{i}\equiv r_{3}+i \cdot\frac{q}{d}\
(\mathrm{mod}\ q) ,\ 0<a_{i}\leq q,\ i=0,1, \cdots ,d-1.$$ By Lemma
2.3, there are precisely $ \varphi(n)/\varphi(n/d)\geq\varphi(d)$
numbers in $\{a_{0},a_{1},\cdots,a_{d-1}\}$ which are coprime to $ q
$. Since $\varphi(d)>2$, there exist distinct $ a_{i_{0}} $, $
a_{j_{0}} $ such that $ a_{i_{0}}\neq r_3 $, $ a_{j_{0}}\neq r_3 $,
 and $
\mathrm{gcd}(a_{i_{0}},q)= \mathrm{gcd}(a_{j_{0}},q)=
  1 $. Applying (6) with $a=r_3$, $a=a_{i_{0}}$ and $a=a_{j_{0}}$, we
get
$$ \sum_{m\in M} \frac{f(r_{3}m)}{m} =f(r_3)+\sum_{r_{3}m\equiv r_{1}\
(\mathrm{mod}\ q) \atop m\in M} \frac{f(r_{1})}{m}
+\sum_{r_{3}m\equiv r_{2}\ (\mathrm{mod}\ q)\atop m\in M}
\frac{f(r_{2})}{m}=0 ,
 $$
 $$ \sum_{m\in M} \frac{f(a_{i_{0}}m)}{m}=f(a_{i_{0}})+\sum_{a_{i_{0}}m\equiv r_{1}\
(\mathrm{mod}\ q) \atop m\in M} \frac{f(r_{1})}{m}
+\sum_{a_{i_{0}}m\equiv r_{2}\ (\mathrm{mod}\ q)\atop m\in M}
\frac{f(r_{2})}{m}=0 ,
 $$
 $$ \sum_{m\in M} \frac{f(a_{j_{0}}m)}{m}=f(a_{j_{0}})+\sum_{a_{j_{0}}m\equiv r_{1}\
(\mathrm{mod}\ q) \atop m\in M} \frac{f(r_{1})}{m}
+\sum_{a_{j_{0}}m\equiv r_{2}\ (\mathrm{mod}\ q)\atop m\in M}
\frac{f(r_{2})}{m}=0 .
 $$
Observe that for every $m\in M$, we have $$ r_{3}m\equiv r_{i}\
(\mathrm{mod}\ q)\Longleftrightarrow a_{i_{0}}m\equiv r_{i}\
(\mathrm{mod}\ q)\Longleftrightarrow a_{j_{0}}m\equiv r_{i}\
(\mathrm{mod}\ q), \quad i=1,2.$$ It follows that $
f(r_{3})=f(a_{i_{0}})= f(a_{j_{0}})\neq 0 $, which contradicts to
our assumptions.

 Now we consider the case
$\varphi(d)\leq 2$, that is $ d=2,3,4,6 $.

\,{\bf Case 1.}\,  $ d=2 $. First we consider the subcase of $ 2 \|
q $. If $ 2 \| q $, we choose $u_0$  to be the smallest positive
integer such that $2^{u_0}\equiv 1\
 (\mathrm{mod}\ \frac{q}{2})$.
It is easy to see that $\varepsilon(r_1,2)=\varepsilon(r_2,2)=2$,
applying (5) with $p=2$, we get
\begin{equation}
f(r_1)+f(r_2)=0.
\end{equation}

Now we prove the following Claim:

{\bf Claim:} If there are positive integers $k$ and $c\in J$ such
that $r_1\equiv2^{k}c \ (\mathrm{mod}\ q)$, then $f(c)\neq0$.

 Otherwise, if $f(c)=0$, applying (6) with $a=c$ we have
\begin{equation}\label{equ_9}
 \sum_{m\in M} \frac{f(cm)}{m}=\sum_{cm\equiv r_{1}\ (\mathrm{mod}\ q) \atop
m\in M} \frac{f(r_{1})}{m} +\sum_{cm\equiv r_{2}\ (\mathrm{mod}\
q)\atop m\in M} \frac{f(r_{2})}{m}=0.
 \end{equation}
Since $\mathrm{gcd}(r_{1}, r_{2}, q)=2$, then $cm\equiv r_{i}\
(\mathrm{mod}\ q) $, $i=1,2$ can occur only when $m=2^{x}$ for some
positive integer $x$. If the congruence $r_2\equiv2^{x}c \
(\mathrm{mod}\ q)$ has no solution $x$, then by (8), we have
$$f(r_1)\sum_{cm\equiv r_{1}\ (\mathrm{mod}\ q) \atop m\in M}
\frac{1}{m}=0,$$ and so $f(r_1)=0$, a contradiction. If the
congruence $r_2\equiv2^{x}c \ (\mathrm{mod}\ q)$ has solutions, we
take $l$ to be the smallest positive integer solution, then all
positive solutions can be expressed as $l+tu_0$, $t=0,1,2,\cdots$.
  Let $k_0$ be the smallest positive
integer solution of the congruence $r_1\equiv2^{k}c \ (\mathrm{mod}\
q)$. Then (8) becomes
\begin{equation}
\frac{f(r_1)}{2^{k_0}}\frac{1}{1-2^{-u_0}}+\frac{f(r_2)}{2^{l}}\frac{1}{1-2^{-u_0}}=0.
\end{equation}
Combining (7) and (9), we get $k_0=l$, which implies that $r_1\equiv
r_2 \ (\mathrm{mod}\ q)$, a contradiction. We have proved the Claim.

For given positive integers $n,a,$ and $i$ with
$\mathrm{gcd}(a,q)=1$, since $ 2 \| q $, then the congruence
$2^{n}a\equiv2^{i}x_i\ (\mathrm{mod}\ q)$ has precisely one solution
$x_{i}$ such that $0<x_i<q$ and $\mathrm{gcd}(x_i,q)=1$. On the
other hand, if $1\leq i<j\leq u_0$, then $x_{i}\neq x_{j}$. Indeed,
if $2^{i}x_i\equiv 2^{j}x_i\equiv 2^{n}a \ (\mathrm{mod}\ q)$, it
follows that $2^{j-i}\equiv1 \ (\mathrm{mod}\ \frac{q}{2})$,
$u_0|j-i$, a contradiction.
 Let $ r_{1}=2^{k}R_{1},\ r_{2}=2^{l}R_{2} $,
 where $ k,l,R_{1},R_{2} $ are positive integers and $ \mathrm{gcd}(R_{1}R_{2},q) = 1 $.
Let $x_i$ be the unique solution of congruence

$$2^{k}R_{1}\equiv2^{i}x_i\ (\mathrm{mod}\ q),\ 0<x_i<q,\ \mathrm{gcd}(x_i,q)=1,\ i=1,2,\cdots,u_0.$$
By the Claim and the above arguments we have $f(x_i)\neq0$,
$i=1,2,\cdots,u_0$, $\mathrm{gcd}(x_i,q)=1$ and $x_{i}\neq x_{j}\
(i\neq j)$, and so $u_{0}\leq2$ since we have $f(x)=0$ for $x\in
J\backslash \{r_3,r_4\}$. If $u_{0}=1$, then $q=2$, a contradiction.
If $u_{0}=2$, then $q=6$. Without loss of generality we may assume
that $r_1=2,r_2=4,r_3=1,r_4=5$. Applying (5) with $p=2$ and (6) with
$a=r_3$ and $a=r_4$, we have
\begin{eqnarray*}
\left \{\begin{array}{l} f(r_1)+f(r_2)=0,\\
f(r_3)+\frac{f(r_1)}{2}\frac{1}{1-2^{-2}}+\frac{f(r_2)}{4}\frac{1}{1-2^{-2}}=0,\\
f(r_4)+\frac{f(r_1)}{4}\frac{1}{1-2^{-2}}+\frac{f(r_2)}{2}\frac{1}{1-2^{-2}}=0.
\end{array}\right.
\end{eqnarray*}
Hence
$$f(r_2)=-f(r_1),f(r_3)=-\frac{1}{3}f(r_1),f(r_4)=\frac{1}{3}f(r_1).$$
By Lemma 2.1 we get
$$T=\frac{1}{3}f(r_1)\sum_{n=0}^{\infty}\{\frac{3}{6n+2}-\frac{3}{6n+4}-\frac{1}{6n+1}+\frac{1}{6n+5}\}$$
$$=\frac{2}{3}f(r_1)\sum_{n=0}^{\infty}\frac{36n^{2}+36n-1}{(6n+1)(6n+2)(6n+4)(6n+5)}=0.$$

Next we consider the case that $ q=4 $, without loss of generality
we may assume that $r_1=2,r_2=4,r_3=1,r_4=3$. Applying (5) with
$p=2$ and (6) with $a=r_3$ and $a=r_4$, we have
\begin{eqnarray*}
\left \{\begin{array}{l} f(r_1)+3f(r_2)=0,\\
f(r_3)+\frac{f(r_1)}{2}+\frac{f(r_2)}{4}\frac{1}{1-\frac{1}{2}}=0,\\
f(r_4)+\frac{f(r_1)}{2}+\frac{f(r_2)}{4}\frac{1}{1-\frac{1}{2}}=0.
\end{array}\right.
\end{eqnarray*}
Hence
$$f(r_1)=-3f(r_2),f(r_3)=f(r_2),f(r_4)=f(r_2).$$
By Lemma 2.1 we have
$$T=f(r_2)\sum_{n=0}^\infty\{\frac{-3}{4n+2}+\frac{1}{4n+4}+\frac{1}{4n+1}+\frac{1}{4n+3}\}$$
$$=f(r_2)\sum_{n=0}^\infty
  \frac{16n^2+12n-1}{(4n+1)(4n+2)(4n+3)(4n+4)}=0.$$

Now we deal with the case $ 4|q$ and $q>4$. Since
$d=\mathrm{gcd}(r_{1},r_{2}, q)=2$, $ 4|q$, without loss of
generality we may assume that $q=2^{\alpha_{0}}Q$, $r_{1}=2R_{1}$,
and $r_{2}=2^{l}R_{2} $,
 where $Q,l,R_1,R_2,\alpha_{0}$ are positive integers, $\alpha_{0}\geq2$,  $2\nmid Q$, $l\geq1$ and $\mathrm{gcd}(R_{1}R_{2}, q)=1$.
Let $$a_{i}\equiv r_{3}+i\cdot \frac{q}{2}\ (\mathrm{mod}\ q) ,\
0<a_{i}\leq q ,\ i=0,1 .$$ Since $ 4| q$, then
$\mathrm{gcd}(a_{0}a_{1},q)=1$. Note that $$\{m\in M|\ ma_0\equiv
r_i \ (\mathrm{mod}\ q)\}=\{m\in M|\ ma_1\equiv r_i \ (\mathrm{mod}\
q)\}, \ i=1,2.$$ Applying (6) with $a=a_0$ and $a=a_1$ we get
$$f(a_0)=f(a_1).$$
Since $a_0=r_3$ and $f(x)=0$ for $x\in J\backslash \{r_3,r_4\}$, we
have
 $a_1=r_4$ and $f(r_{3})=f(r_{4})$. Note
that
$$M_1=\{m\in M|\ R_{1}m\equiv r_1=2R_{1}\ (\mathrm{mod}\ q)
\}=\{2\}$$ and
$$M_2=\{m\in M|\ R_{1}m\equiv r_2\ (\mathrm{mod}\ q) \}=\{2^{n}\in M|\ R_{1}2^{n}\equiv r_2\ (\mathrm{mod}\ q)\}.$$
If the congruence $r_{2}\equiv 2^{x}R_{1}\ (\mathrm{mod}\ q)$ has no
solution, then by applying (6) with $a=R_1$ we get $
f(R_{1})+\frac{f(r_{1})}{2}=0 $, and so $
f(R_{1})=-\frac{f(r_{1})}{2} \neq0$, it follows that  $
f(R_{1})=f(r_{4})=f(r_{3})=-\frac{f(r_{1})}{2}$. Since
$f(r_1)+f(r_2)+f(r_3)+f(r_4)=0$, so $f(r_2)=0$, a contradiction. Now
we assume that $l^{'}$ is the smallest positive solution of the
congruence $ r_{2}\equiv 2^{x}R_{1}\ (\mathrm{mod}\ q)$. Let $u_0$
be the smallest positive integer such that $2^{u_0}\equiv 1\
 (\mathrm{mod}\ Q)$.
We consider the following four subcases.

(i) If $ f(R_{1})=0 $ and $ l\geq \alpha_{0}$. Applying (5) with
$p=2$ and (6) with $a=R_1$, we get
\begin{eqnarray*}
\left \{\begin{array}{l} f(r_{1})+(\alpha_{0} +1)f(r_{2})=0,\\
\frac{f(r_{1})}{2}+\frac{f(r_{2})}{2^{l^{'}}}\frac{1}{1-2^{-u_{0}}}=0,
\end{array}\right.
\end{eqnarray*}
then $ \alpha_{0} +1=\frac{2^{u_{0}-l^{'}+1}}{2^{u_{0}}-1} $, and so
$\alpha_{0}=l^{'}=u_{0}=1$, a contradiction.

(ii) If $ f(R_{1})=0 $ and $ l<\alpha_{0}$, then $l=l^{'}$ and
$M_2=\{2^{l^{'}}\}$. Applying (5) with $p=2$ and (6) with $a=R_1$,
we get
\begin{eqnarray*}
\left \{\begin{array}{l} f(r_{1})+lf(r_{2})=0,\\
\frac{f(r_{1})}{2}+\frac{f(r_{2})}{2^{l^{'}}}=0,
\end{array}\right.
\end{eqnarray*}
then $ l=\frac{1}{2^{l^{'}-1}}$, and so $ l=l^{'}=1 $ and $
r_{2}\equiv 2R_{1}\equiv r_{1}\ (\mathrm{mod}\ q)$, a contradiction.

(iii) If $ f(R_{1})\neq 0 $ and  $ l\geq \alpha_{0}$. Similarly, we
have
\begin{eqnarray*}
\left \{\begin{array}{l} f(r_{3})=f(r_{4}),\\
f(r_{1})+(\alpha_{0}+1)f(r_{2})=0,\\
f(r_{1})+f(r_{2})+f(r_{3})+f(r_{4})=0,\\
f(r_{3})+\frac{f(r_{1})}{2}+\frac{f(r_{2})}{2^{l^{'}}}\frac{1}{1-2^{-u_{0}}}=0,
\end{array}\right.
\end{eqnarray*}
then $2^{l^{'}-1}(1-2^{-u_0})=1$, and so $u_0=1,l^{'}=2,q=4$, a
contradiction.

(iv) If $ f(R_{1})\neq 0 $ and $ l<\alpha_{0}$, then $M_1=\{2\}$ and
$M_2=\{2^{l^{'}}\}$. Similarly, we have
\begin{eqnarray*}
\left \{\begin{array}{l} f(r_{3})=f(r_{4}),\\ f(r_{1})+lf(r_{2})=0,\\
f(r_{1})+f(r_{2})+f(r_{3})+f(r_{4})=0,\\
f(r_{3})+\frac{f(r_{1})}{2}+\frac{f(r_{2})}{2^{l^{'}}}=0,
\end{array}\right.
\end{eqnarray*}
then $2^{l{'}}=2$, $l^{'}=1 $, $l=l^{'}=1 $ by the definition of $l$
and $l^{'}$, and so $ r_{2}\equiv 2R_{1}=r_{1}\ (\mathrm{mod}\ q)$,
again a contradiction.

\,{\bf Case 2.}\,  $ d=3 $. Let
$$ a_{j}\equiv
r_{3}+j\frac{q}{3}\ (\mathrm{mod}\ q),\  0< a_{j}\leq q,\ j=0,1,2,$$
and let $$M_{ij}=\{m\in M \ |\ ma_j\equiv r_i\ (\mathrm{mod}\ q)\},\
i=1,2,\ j=0,1,2.$$

If $ 9| q $, then $$M_{10}=M_{11}=M_{12},\ M_{20}=M_{21}=M_{22}.$$
Applying (6) with $a=a_0,a_1$ and $ a_2$, we have
$$f(r_{3})=f(a_{0})=f(a_{1})=f(a_{2}),$$ and $a_0,a_1,a_2$ are
distinct, which contradicts to the fact that $f(x)=0$ for $x\in
J\backslash \{r_3,r_4\}$.

If $ 3 \| q $, then by Lemma 2.3 we can choose $a_{j_0}\in
\{a_1,a_2\}$ such that $\mathrm{gcd}(a_{j_0},q)=1$. Similarly, we
have $$f(a_{0})=f(a_{j_0}),$$ so $a_{j_0}=r_4$ and
$f(r_{3})=f(r_{4})$. Applying (5) with $p=3$, we get
$f(r_{1})+f(r_{2})=0$. Combining with $f(r_{3})=f(r_{4})$,
$f(r_{1})+f(r_{2})+f(r_{3})+f(r_{4})=0$, we have $ f(r_{3})=0 $, a
contradiction.

The cases $d=4$ and $d=6$ are similar to  $ d=3 $,  and we omit the
details. This completes the proof.\end{proof}

\begin{proposition}
 Suppose that
$$
T=\sum_{n=0}^\infty
  \frac{\alpha n^2+\beta
  n+\gamma}{(qn+r_1)(qn+r_2)(qn+r_3)(qn+r_4)}=0,
$$
where $r_1, r_2, r_3, r_4 $ are distinct positive integers $ \leq q
$ and $\alpha, \beta, \gamma \in \overline{\mathbb{Q}}$. Suppose
$\gcd(\alpha n^2+\beta
  n+\gamma, (qn+r_1)(qn+r_2)(qn+r_3)(qn+r_4))=1$
, $\mathrm{gcd}(r_{1}, r_{2},r_{3}, r_{4}, q)=1$ and $\Phi_q$ is
irreducible over
  $\mathbb{Q}(\alpha,\beta, \gamma)$. Then $\rho\neq1$.
\end{proposition}

\begin{proof}  Suppose $\rho=1$. Without loss of
generality we may assume that $ \mathrm{gcd}(r_{i}, q)>1 $,
$i=1,2,3$, and $ \mathrm{gcd}(r_{4}, q)=1 $.

First we consider the case that there exist distinct integers $
r_{i}, r_{j}\in\{r_{1}, r_{2}, r_{3}\}$, such that $
\varphi(\mathrm{gcd}(r_{i}, r_{j}, q))>1$, say $\varphi
(\mathrm{gcd}(r_{2} ,r_{3},q))>1$. Let
$$a_{i}=1+i\cdot\frac{q}{\mathrm{gcd}(r_{2}
,r_{3},q)},\ i=0,1,\cdots,\mathrm{gcd}(r_{2} ,r_{3},q)-1.$$ By Lemma
2.3, we may choose $a_{i_{0}}$ such that $a_{i_{0}}\neq
 1$ and $\mathrm{gcd}(a_{i_{0}},q)=1 $. Applying Lemma 2.2 with
 $k=a_{i_0}$, we have
\begin{equation}\label{equ_9}
\sum_{n=1}^{\infty} \frac{f(a_{i_{0}}n)}{n}=
\sum_{n=1}^{\infty}{\{\frac{f(r_{1})}{nq+r_{1}^{'}}+\frac{f(r_{2})}{nq+r_{2}}+\frac{f(r_{3})}{nq+r_{3}}+\frac{f(r_{4})}{nq+r_{4}^{'}}\}}=0,
\end{equation}
where $r_{1}^{'}\equiv a_{i_{0}}^{-1}r_{1}$, $r_{4}^{'}\equiv
 a_{i_{0}}^{-1}r_{4}\ (\mathrm{mod}\ q )$ and $0< r_{1}^{'},r_{4}^{'}< q $.
Obviously $r_{4}\neq r_{4}^{'}$ since $a_{i_0}\not \equiv 1 \
(\mathrm{mod}\ q )$ and $\mathrm{gcd}(r_{4}, q)=1$. Subtracting $T$
from (10), we obtain
$$T^{'}=\sum_{n=1}^{\infty}{\{\frac{f(r_{1})}{nq+r_{1}^{'}}-\frac{f(r_{1})}{nq+r_{1}}+\frac{f(r_{4})}{nq+r_{4}^{'}}}-\frac{f(r_{4})}{nq+r_{4}}\}=0.$$
If $r_{1}=r_{1}^{'}$, then $T^{'}=f(r_4)
\sum_{n=1}^{\infty}\{\frac{1}{r_{4}^{'}}-\frac{1}{r_4}\}\neq 0$, a
contradiction. If $r_{1}\neq r_{1}^{'}$, then there are precisely
two integers $r_{4},r_{4}^{'}$ in
$\{r_1,r_{1}^{'},r_{4},r_{4}^{'}\}$ which are coprime to $q$. By
Proposition 3.2 we have

$$T^{'}=\sum_{n=0}^\infty\{\frac{-3}{4n+2}+\frac{1}{4n+4}+\frac{1}{4n+1}+\frac{1}{4n+3}\},\ \ q=4,\ $$
or
 $$T^{'}=\sum_{n=0}^\infty\{\frac{3}{6n+2}+\frac{-3}{6n+4}+\frac{-1}{6n+1}+\frac{1}{6n+5}\}, \ q=6.$$
The first equality is impossible since $q>4$. If the second equality
holds, then $q=6$, $\{r_2,r_3\}=\{3,6\}$ and $3\not| r_1$. Applying
(5) with $p=3$, we get $f(r_2)+f(r_3)=0$, which implies that
$f(r_1)+f(r_4)=0$. But in the second equality we have
$f(r_1)=3f(r_4)$ or $f(r_1)=-3f(r_4)$, a contradiction.

Now we assume that $ \varphi{(\mathrm{gcd}(r_{i},r_{j},q))}\leq1 $
for all distinct integers $ r_{i}, r_{j}\in{\{r_{1}, r_{2},
r_{3}\}}$, then $
\mathrm{gcd}(r_{1},r_{2},q)=\mathrm{gcd}(r_{1},r_{3},q)=\mathrm{gcd}(r_{2},r_{3},q)=2
$.

(i) If $ 2\|q $, then applying (5) with $p=2$, we have $
f(r_{1})+f(r_{2})+f(r_{3})=0 $, and so $f(r_{4})=0$, a
contradiction.

(ii) If $ 4|q $, let $ a_1=1+\frac{q}{2}$, then $a_{1}\neq1$ ,
$\mathrm{gcd}(a_{1},q)=1$, and $a_{1}r_{i}\equiv r_{i}\
(\mathrm{mod}\ q)$, $i=1,2,3 $. Applying Lemma 2.2 with
 $k=a_{1}$, we have

\begin{equation}
\sum_{n=1}^{\infty}\frac{f(a_{1}n)}{n}=\sum_{n=0}^{\infty}{\{\frac{f(r_{1})}{nq+r_{1}}+\frac{f(r_{2})}{nq+r_{2}}+\frac{f(r_{3})}{nq+r_{3}}+\frac{f(r_{4})}{nq+r_{4}^{'}}\}}=0
,\end{equation}
 where $r_{4}^{'}\equiv r_{4}+\frac{q}{2}\
(\mathrm{mod}\ q)$ and $0<r_{4}^{'}<q$. Subtracting $T$ from (11),
we obtain
$\sum_{n=0}^{\infty}{\{\frac{f(r_{4})}{nq+r'_{4}}-\frac{f(r_{4})}{nq+r_{4}}}\}=0,$
which  contradicts to $r_{4}^{'}\neq r_{4}$. The proof is complete.
\end{proof}

\begin{proposition}
  Suppose that
$$
T=\sum_{n=0}^\infty
  \frac{\alpha n^2+\beta
  n+\gamma}{(qn+r_1)(qn+r_2)(qn+r_3)(qn+r_4)}=0,
$$
where $r_1, r_2, r_3, r_4 $ are distinct positive integers $ \leq q
$ , and $\alpha, \beta, \gamma \in \overline{\mathbb{Q}}$. Suppose
$\gcd(\alpha n^2+\beta
  n+\gamma, (qn+r_1)(qn+r_2)(qn+r_3)(qn+r_4))=1$
, $\mathrm{gcd}(r_{1}, r_{2},r_{3}, r_{4}, q)=1$ and $\Phi_q$ is
irreducible over
  $\mathbb{Q}(\alpha,\beta, \gamma)$. Then $\rho\neq0$.
\end{proposition}
\begin{proof} Suppose $\rho=0$. We divide the
proof into two cases.

\,{\bf Case 1.}\, There exist distinct integers
$r_{i},r_{j},r_{k}\in\{r_{1},r_{2},r_{3},r_{4}\}$ such that $
\mathrm{gcd}(r_{i},r_{j},r_{k},q)>1$, say,
$d=\mathrm{gcd}(r_{1},r_{2},r_{3},q)>1$. Let
$$a_{i}=1+i\cdot\frac{q}{d},\ i=0,1,\cdots,d-1.$$ If $\varphi(d)>1$,  we may choose $a_{i_{0}}\in\{a_0,a_1,\cdots,a_{d-1}\}$ such
that $a_{i_{0}}\neq1$ and $a_{i_{0}}\in J$ by Lemma 2.3. Note that
$$a_{i_{0}}r_{j}\equiv r_{j}\ (\mathrm{mod}\ q),\ j=1,2,3.$$
Applying Lemma 2.2 with
 $k=a_{i_0}$, we obtain
\begin{equation}
\sum_{n=1}^{\infty}\frac{f(a_{i_{0}}n)}{n}=
\sum_{n=0}^{\infty}{\{\frac{f(r_{1})}{nq+r_{1}}+\frac{f(r_{2})}{nq+r_{2}}+\frac{f(r_{3})}{nq+r_{3}}+\frac{f(r_{4})}{nq+r_{4}^{'}}}\}=0,\end{equation}
where $r_{4}^{'}\equiv a_{i_{0}}^{-1}r_{4}\ (\mathrm{mod}\ q)$ and
$0<r_{4}^{'}<q $. It is easy to check that $r_{4}\neq r_{4}^{'}$
since $\mathrm{gcd}(r_4,d,q)=1$. Subtracting the second equality of
(12) from $T$, we have
$\sum_{n=0}^{\infty}{\{\frac{f(r_{4})}{nq+r_{4}}-\frac{f(r_{4})}{nq+r_{4}^{'}}}\}=0
$, which is impossible since $r_{4}\neq r_{4}^{'}$.

If $\varphi(d)=1$, then $d=2=\mathrm{gcd}(r_{1},r_{2},r_{3},q)$.

(i) If $ 2\|q $, applying (5) with $p=2$ we get $
f(r_{1})+f(r_{2})+f(r_{3})=0 $, and so $f(r_{4})=0$, a
contradiction.

(ii) If $ 4|q $, we take $ b_{1}=1+\frac{q}{2}$, then
$\mathrm{gcd}(b_{1},q)=1$, and $b_{1}r_{i}\equiv r_{i}\
(\mathrm{mod}\ q)$, $i=1,2,3 $. Applying (6) with $a=b_1$ we have

\begin{eqnarray*}
\sum_{n=1}^{\infty}\frac{f(b_{1}n)}{n}=\sum_{n=0}^{\infty}{\{\frac{f(r_{1})}{nq+r_{1}}+\frac{f(r_{2})}{nq+r_{2}}+\frac{f(r_{3})}{nq+r_{3}}+\frac{f(r_{4})}{nq+r_{4}^{'}}\}}=0
,\end{eqnarray*}
 where $r_{4}^{'}\equiv b_{1}^{-1}r_{4}\
(\mathrm{mod}\ q)$ and $0<r_{4}^{'}< q$. Similarly, $r_{4}^{'}\neq
r_4$. Subtracting the above second equality from $T$, we obtain
$\sum_{n=0}^{\infty}{\{\frac{f(r_{4})}{nq+r_{4}}-\frac{f(r_{4})}{nq+r_{4}^{'}}}\}=0
$, which is also impossible since $r_{4}\neq r_{4}^{'}$.

\,{\bf Case 2.}\, If  $ \mathrm{gcd}(r_{i},r_{j},r_{k},q)=1$ for all
distinct $r_{i},r_{j},r_{k}\in{\{r_{1},r_{2},r_{3},r_{4}\}}$, then
there exist distinct $r_{i},r_{j}\in{\{r_{1},r_{2},r_{3},r_{4}\}}$
such that $ \varphi(\mathrm{gcd}(r_{i},r_{j},q))>1$, say, $
\varphi(\mathrm{gcd}(r_{1},r_{2},q))>1$. Otherwise, we would have
$gcd(r_i, \, r_j, q)=1$ or 2 for all distinct $r_i, \, r_j \in\{r_1,
\, r_2, \, r_3, \, r_4\}$,  and it would  follow that
$\gcd(r_1,q)\gcd(r_2,q)\gcd(r_3,q)\gcd(r_4,q)$ has only one prime
divisor 2 by the argument of the paragraph above Proposition 3.2,
and this would mean that  $\gcd(r_1, \, r_2, \, r_3, \, r_4,\, q)=2$
since $\rho=0$, contradicting  our assumptions. Let
$$c_{i}=1+i\cdot\frac{q}{\mathrm{gcd}(r_{1},r_{2},q)},\ i=0,1, \cdots,
\mathrm{gcd}(r_{1},r_{2},q)-1.$$ By Lemma 2.3 we can choose
$c_{i_{0}}$ such that $c_{i_{0}}\neq 1$ and
$\mathrm{gcd}(c_{i_{0}},q)=1$. Note that $c_{i_{0}}r_{j}\equiv
r_{j}\ (\mathrm{mod}\ q)$, $j=1,2$.  Similarly, applying Lemma 2.2
with
 $k=c_{i_{0}}$
and subtracting, we obtain
$$T^{'}=\sum_{n=0}^{\infty}{\{\frac{f(r_{3})}{nq+r_{3}}-\frac{f(r_{3})}{nq+r_{3}^{'}}}+\frac{f(r_{4})}{nq+r_{4}}-\frac{f(r_{4})}{nq+r_{4}^{'}}\}=0,$$
where $r_{3}^{'}\equiv c_{i_{0}}^{-1}r_{3}\pmod{q}$,
$r_{4}^{'}\equiv c_{i_{0}}^{-1}r_{4}\ (\mathrm{mod}\ q)$,
$0<r_{3}^{'},r_{4}^{'}<q $. Note that $r_{3}\neq r_{3}^{'}$,
$r_{4}\neq r_{4}^{'}$,
$\mathrm{gcd}(r_{3},q)=\mathrm{gcd}(r_{3}^{'},q)$ and
$\mathrm{gcd}(r_{4},q)=\mathrm{gcd}(r_{4}^{'},q)$.

(i) If $\mathrm{gcd}(r_{3},r_4,q)=1$. Since $
\mathrm{gcd}(r_{3},q)>1$ and $ \mathrm{gcd}(r_{4},q)>1$, without
loss of generality, we may assume that $ \mathrm{gcd}(r_{3},q)>2$,
that is $\varphi(\mathrm{gcd}(r_{3},q))>1$. Let
$$d_{j}=1+j\cdot\frac{q}{\mathrm{gcd}(r_{3},q)},\ j=0,1,\cdots,\mathrm{gcd}(r_{3},q)-1.$$ Similarly, we can choose
$d_{j_{0}}\neq 1$ and $\mathrm{gcd}(d_{j_{0}},q)=1$ such that
$$\sum_{n=0}^{\infty}{\{\frac{f(r_{3})}{nq+r_{3}}-\frac{f(r_{3})}{nq+r_{3}^{'}}+\frac{f(r_{4})}{nq+r_{4}^{''}}-\frac{f(r_{4})}{nq+r_{4}^{'''}}\}}=0,$$
where $r_{4}^{''}\equiv d_{j_{0}}^{-1}r_{4}\pmod{q}$ ,
$r_{4}^{'''}\equiv d_{j_{0}}^{-1}r_{4}^{'}$ $(\mathrm{mod}\ q)$,
$0<r_{4}^{'},r_{4}^{''},r_{4}^{'''}< q $. It follows that
$$T^{''}=\sum_{n=0}^{\infty}{\{\frac{f(r_{4})}{nq+r_{4}}-\frac{f(r_{4})}{nq+r_{4}^{'}}-\frac{f(r_{4})}{nq+r_{4}^{''}}+\frac{f(r_{4})}{nq+r_{4}^{'''}}\}}=0.$$
 Note that $r_4^{'}\neq
r_{4}^{'''}$, $r_4\neq r_{4}^{''}$, $r_4\neq r_4'$, $r_4^{''}\neq
r_4^{'''}$ and
$\mathrm{gcd}(r_{4},q)=\mathrm{gcd}(r_{4}^{'},q)=\mathrm{gcd}(r_{4}^{''},q)=\mathrm{gcd}(r_{4}^{'''},q)$
since $\gcd(c_{i_0}d_{j_0}, q)=1$. Now we have
$$T^{''}=\frac{f(r_4)}{\mathrm{gcd}(r_{4},q)}\sum_{n=0}^{\infty}\{\frac{1}{nq^{'}+a}+\frac{-1}{nq^{'}+b}+\frac{-1}{nq^{'}+c}+\frac{1}{nq^{'}+e}\}=0,$$
where $q^{'}=\frac{q}{\mathrm{gcd}(r_4,
q)},a=\frac{r_4}{\mathrm{gcd}(r_4,
q)},b=\frac{r_{4}^{'}}{\mathrm{gcd}(r_4,
q)},c=\frac{r_{4}^{''}}{\mathrm{gcd}(r_4,
q)},e=\frac{r_{4}^{'''}}{\mathrm{gcd}(r_4, q)}$. Obviously all of
$a,b,c,e$ are coprime to $q'$ and $\Phi_{q^{'}}$ is irreducible over
$\mathbb{Q}$ .

It is easy to check that $a,b,c,e$ are distinct. Otherwise, since
$a\neq b$, $a\neq c$, $c\neq e$, $b\neq e$, we have $a=e$ or $b=c$.
If $a=e$ and $b=c$ both hold, then
$$T^{''}=\frac{f(r_4)}{\mathrm{gcd}(r_{4},q)}\sum_{n=0}^{\infty}\{\frac{2}{nq^{'}+a}+\frac{-2}{nq^{'}+b}\}\neq
0,$$ which is a contradiction. If $a=e$ and $b\neq c$, then by
Theorem C we have
$$T^{''}=\frac{f(r_4)}{\mathrm{gcd}(r_{4},q)}\sum_{n=0}^{\infty}\{\frac{2}{nq^{'}+a}+\frac{-1}{nq^{'}+b}+\frac{-1}{nq^{'}+c}\}\neq 0,$$
which is also a contradiction. Similarly, the case that $a\neq e$
and $b=c$ is impossible. Therefore $T^{''}=0$ is impossible by
Proposition 3.1.

(ii) If $\mathrm{gcd}(r_{3},r_4,q)>1$, applying (5) with some prime
$p$ satisfying $p|\mathrm{gcd}(r_{3},r_4,q)$, we get $$vf(r_{3})+u
f(r_{4})=0,$$ where $u,v$ are positive rational numbers, then we may
re-write $T^{'}$ as

$$T^{'}=\frac{f(r_4)}{\mathrm{gcd}(r_{3}, r_4, q)}\sum_{n=0}^{\infty}
{\{\frac{-\frac{u}{v}}{nq^{'}+a}+\frac{\frac{u}{v}}{nq^{'}+b}+\frac{1}{nq^{'}+c}+\frac{-1}{nq^{'}+e}\}},$$
where $q^{'}=\frac{q}{\mathrm{gcd}(r_{3}, r_4, q)},\
a=\frac{r_3}{\mathrm{gcd}(r_{3}, r_4, q)},\
b=\frac{r_{3}^{'}}{\mathrm{gcd}(r_{3}, r_4, q)},\
c=\frac{r_4}{\mathrm{gcd}(r_{3}, r_4, q)},\
e=\frac{r_{4}^{'}}{\mathrm{gcd}(r_{3}, r_4, q)}$. It is easy to see
that $\mathrm{gcd}(a,q^{'})=\mathrm{gcd}(b,q^{'}),\
\mathrm{gcd}(c,q^{'})=\mathrm{gcd}(e,q^{'})$, and
$\mathrm{gcd}(a,c,q^{'})=1$. Note that $\Phi_{q^{'}}$ is irreducible
over $\mathbb{Q}$ and $a, b, c, e$ are distinct integers by the same
arguments as above.

By Proposition 3.1, we have either $\mathrm{gcd}(a,q^{'})>1$ or
$\mathrm{gcd}(c,q^{'})>1$.

  If precisely
one of $\mathrm{gcd}(a,q^{'}),\mathrm{gcd}(c,q^{'})$ is 1, then
without loss of generality we may assume that
$\mathrm{gcd}(c,q^{'})=1$ and $\mathrm{gcd}(a,q^{'})>1$. Then by
Proposition 3.2 we have that
\begin{equation}T^{'}=\sum_{n=0}^\infty\{\frac{-3}{4n+2}+\frac{1}{4n+4}+\frac{1}{4n+1}+\frac{1}{4n+3}\}=0,\ q'=4,
\end{equation}
or
\begin{equation}T^{'}=\sum_{n=0}^\infty\{\frac{3}{6n+2}+\frac{-3}{6n+4}+\frac{-1}{6n+1}+\frac{1}{6n+5}\}=0,\
q'=6. \end{equation} (13) is impossible since $\{-3,1,1,1\}\neq
\{-\frac{hu}{v},\frac{hu}{v},h,-h\}$ for all $ h\in \mathbb{Q}$. If
(14) holds, then $q^{'}=6$, $\{a,b\}=\{2,4\}$ and $\frac{u}{v}=3$,
that is $q=6\mathrm{gcd}(r_3,r_4,q)$. Note that $\gcd(r_3,
q)=\gcd(r_3, r_4, q)\gcd(a,q')$ and $\gcd(\gcd(r_1, r_2, q),
\gcd(r_3, q))=1$. It follows that
$$\mathrm{gcd}(r_1,r_2,q)\mathrm{gcd}(a,q^{'})|6.$$
 Since
$\mathrm{gcd}(a,q^{'})>1$ and $\varphi(\mathrm{gcd}(r_1,r_2,q))>1$,
then $\mathrm{gcd}(a,q^{'})=2$ and $\mathrm{gcd}(r_1,r_2,q)=3$. If
$\mathrm{gcd}(r_3,r_4,q)\neq2$, then
$\varphi(\mathrm{gcd}(r_3,r_4,q))>1$. Similarly, using the same
argument as above we  obtain that $\mathrm{gcd}(r_3,r_4,q)=3$, a
contradiction. If $\mathrm{gcd}(r_3,r_4,q)=2$, then $q=12$ and
$\{r_1,r_2\}=\{3,9\}$. Applying (5) with $p=3$ we have
$f(r_1)+f(r_2)=0$. It follows that $f(r_3)+f(r_4)=0$ which
contradicts with $f(r_3)=-3f(r_4)$ since $\frac{u}{v}=3$.

If $\mathrm{gcd}(a,q^{'})>1,\mathrm{gcd}(c,q^{'})>1$. Since
$\mathrm{gcd}(a,c,q^{'})=1$,
$\mathrm{gcd}(a,q^{'})=\mathrm{gcd}(b,q^{'})$ and
$\mathrm{gcd}(c,q^{'})=\mathrm{gcd}(e,q^{'})$, then one of $\gcd(a,
b, q')$ and $\gcd(c,e, q')$ is larger than 2, say, $\gcd(a, b,
q')>2$. Let
$$l_j =1+j\cdot\frac{q'}{\gcd(a,b, q')}, j=1, \cdots, \frac{q'}{\gcd(a,b, q')}-1.$$
Similarly, we can choose $l_{j_0}\neq1$ and $\gcd(l_{j_0}, q')=1$
such that
$$\sum_{n=0}^{\infty}{\{\frac{-\frac{u}{v}}{nq^{'}+a}+\frac{\frac{u}{v}}{nq^{'}+b}+\frac{1}{nq^{'}+c'}+\frac{-1}{nq^{'}+e'}\}}=0,$$
where $c'\equiv l_{j_0}^{-1}c, e'\equiv l_{j_0}^{-1}e$, $0<c',
e'<q'$. It follows that
$$T_1=\sum_{n=0}^{\infty}{\{\frac{1}{nq^{'}+c}+\frac{1}{nq^{'}+e'}-\frac{1}{nq^{'}+c'}-\frac{1}{nq^{'}+e}\}}=0,$$
$c\neq c', e\neq e', \gcd(c, q')=\gcd(e, q')=\gcd(c'
q')=\gcd(e',q')>1$. The remaining argument is the same line as in
Case 2 (i). This completes the proof.\end{proof}

{\bf Proof of Theorem 1.1:} By the above propositions 3.1-3.4, we
have proven Theorem 1.1.

{\bf Acknowledgement:} The authors wish to thank the referee for
helpful comments on this paper.

\end{document}